\documentclass[12pt]{amsart}
\usepackage{amsmath,amssymb,txfonts}
\usepackage{txfonts}

      \newtheorem{theorem}{Theorem}[section]

      \newtheorem{lemma}[theorem]{Lemma}

      \newcommand{\ct}[1]{\langle {#1}\rangle \lower.3ex\hbox{$_{t}$}}
      \newcommand{\lt}[1]{[ {#1}] \lower.3ex\hbox{$_{t}$}}

\begin{document}

\title[{\tiny Prescribing capacitary curvature measures on planar convex domains}]{Prescribing capacitary curvature measures on planar convex domains}

\author{J. Xiao}
\address{Department of Mathematics and Statistics, Memorial University of Newfoundland, St. John's, NL A1C 5S7, Canada}
\email{jxiao@math.mun.ca}
\thanks{This project was in part supported by NSERC of Canada.}


\date{}



\begin{abstract}
For $p\in (1,2]$ and a bounded, convex, nonempty, open set $\Omega\subset\mathbb R^2$ let $\mu_p(\bar{\Omega},\cdot)$ be the $p$-capacitary curvature measure (generated by the closure $\bar{\Omega}$ of $\Omega$) on the unit circle $\mathbb S^1$. This paper shows that such a problem of prescribing $\mu_p$ on a planar convex domain: ``Given a finite, nonnegative, Borel measure $\mu$ on $\mathbb S^1$, find a bounded, convex, nonempty, open set $\Omega\subset\mathbb R^2$ such that $d\mu_p(\bar{\Omega},\cdot)=d\mu(\cdot)$" is solvable if and only if  $\mu$ has centroid at the origin and its support $\mathrm{supp}(\mu)$ does not comprise any pair of antipodal points. And, the solution is unique up to translation. Moreover, if $d\mu_p(\bar{\Omega},\cdot)=\psi(\cdot)\,d\ell(\cdot)$ with $\psi\in C^{k,\alpha}$ and $d\ell$ being the standard arc-length element on $\mathbb S^1$, then $\partial\Omega$ is of $C^{k+2,\alpha}$.    
\end{abstract}
\maketitle


\section{Statement of Theorem \ref{t11}}\label{s1}
\setcounter{equation}{0}

Continuing from \cite{X1} and \cite{Je96a, Je96b, CNSXYZ, X2}, we prove
\begin{theorem}\label{t11} Let $(p,k,\alpha)\in (1,2]\times\mathbb N\times(0,1)$ and $\mu$ be a finite nonnegative Borel measure on the unit circle $\mathbb S^1$ of $\mathbb R^2$. 
\begin{itemize}
\item[(i)] Existence - there is a bounded, convex, nonempty, open subset $\Omega$ of $\mathbb R^2$ such that $d\mu_p(\bar{\Omega},\cdot)=d\mu(\cdot)$ if and only if $\mu$ has centroid at the origin and its support $\mathrm{supp}(\mu)$ does not comprise any pair of antipodal points.

\item[(ii)] Uniqueness - the domain $\Omega$ in (i) is unique up to translation.

\item[(iii)] Regularity - if $d\mu_p(\bar{\Omega},\cdot)=\psi(\cdot)\,d\ell(\cdot)$, $d\ell$ is the standard arc-length element on $\mathbb S^1$, and $0<\psi\in C^{k,\alpha}(\mathbb S^1)$, i.e., its $k$-th derivative $\psi^{(k)}$ is  $\alpha$-H\"older continuous on $\mathbb S^1$, then the boundary $\partial\Omega$ of $\Omega$ is of $C^{k+2,\alpha}$.
\end{itemize}
\end{theorem}
In the above and below, $\mu_p(\bar{\Omega},\cdot)$ is the $p$-capacitary curvature measure on $\mathbb S^1$ - more precisely - if $u$ is the $p$-equilibrium potential $u_{\bar{\Omega}}$ of $\bar{\Omega}$ - the closure of $\Omega$ (cf. \cite{Lewis, CS, CC}), i.e., the unique solution $u=u_{\bar{\Omega}}$ to the boundary value problem (for a model partial differential equation in geometric potential theory over $\mathbb R^2$; see e.g. \cite{Ada1, Ada2, Aro}):
$$
(\mathrm{eq}_{1<p<2})\quad\quad\begin{cases}
\Delta_p u=\hbox{div}(|\nabla u|^{p-2}\nabla u)=0\quad\mathrm{in}\quad \mathbb R^2\setminus\bar{\Omega};\\
u=0\quad\mathrm{on}\quad{\partial\Omega};\\
\underset{|x|\to\infty}\lim u(x)=1,
\end{cases}
$$
or 
$$
(\mathrm{eq}_{p=2})\quad\quad\begin{cases}
\Delta_{p=2} u=\hbox{div}(\nabla u)=0\quad\mathrm{in}\quad \mathbb R^2\setminus{\bar{\Omega}};\\
u=0\quad\mathrm{on}\quad \partial \Omega;\\
 0<\underset{|x|\to\infty}\liminf \Big(\frac{u(x)}{\log|x|}\Big)\le\underset{|x|\to\infty}\limsup \Big(\frac{u(x)}{\log|x|}\Big)<\infty,
\end{cases}
$$
then 
$$
\mu_p(\bar{\Omega},E)=\int_{\mathsf{g}^{-1}(E)}|\nabla u|^p\,d\mathcal{H}^1=\int_{\mathsf{g}^{-1}(E)}|\nabla u_{\bar{\Omega}}|^p\,d\mathcal{H}^1\ \ \forall\ \ \hbox{Borel}\ E\subset\mathbb S^1,
$$
where $d\mathcal{H}^1$ is the standard $1$-dimensional Hausdorff measure on $\partial\Omega$, $\mathsf{g}^{-1}:\mathbb S^1\to\partial\Omega$ is the inverse of the Gauss map $\mathsf{g}:\partial\Omega\to\mathbb S^1$ (which is defined as the outer unit normal vector at $\partial\Omega$), and the non-tangential limit of $\nabla u=\nabla u_{\bar{\Omega}}$ at each point of $\partial\Omega$ exists $\mathcal{H}^1$-almost everywhere with $|\nabla u|=|\nabla u_{\bar{\Omega}}|\in L^p(\partial\Omega,d\mathcal{H}^1)$ (cf. \cite{LN, LN2, Da}), and hence
$$
d\mu_p(\bar{\Omega},\cdot)=\mathsf{g}_\ast\big(|\nabla u|^p\,d\mathcal{H}^1\big)(\cdot)=\mathsf{g}_\ast\big(|\nabla u_{\bar{\Omega}}|^p\,d\mathcal{H}^1\big)(\cdot)\quad\mathrm{on}\quad\mathbb S^1.
$$

Here it should be pointed out that not only the if-part of Theorem \ref{t11}(i) implies \cite[Theorem 1.2]{CNSXYZ} under $1<p<2=n$ and \cite[Corollary 6.6]{Je96b} under $p=2=n$ due to the fact that $\mathrm{supp}(\mu)$ comprising no any pair of antipodal points amounts to $\mu$ being unsupported on any equator (the intersection of $\mathbb S^1$ with any line passing through the origin) but also Theorems \ref{t11}(ii)\&(iii) under $p\in (1,2)$ have been established in \cite[Theorems 1.2\&1.4]{CNSXYZ}. Our essential contribution to this direction is an establishment of Theorem \ref{t11}(i) and the case $p=2$ of Theorems \ref{t11}(ii)\&(iii).  

Needless to say, Theorem \ref{t11} is not unimportant in that it is nonlinear-potential-theoretic generalization of the classical Minkowski problem in $\mathbb R^2$ concerning the existence, uniqueness and regularity of a planar convex domain with the prescribed curve measure 
$$
d\mu_{cm}=\mathsf{g}_\ast(d\mathcal{H}^1)\ \ \mathrm{on}\ \ \mathbb S^1
$$ 
defined by
$$
\mu_{cm}(E)=\int_{\mathsf{g}^{-1}(E)}\,d\mathcal{H}^1=\mathcal{H}^1\big(\mathsf{g}^{-1}(E)\big)\quad\forall\quad\mathrm{Borel}\ \ E\subset\mathbb S^1.
$$
See e.g. \cite{CY, J, U, K} and their references for an extensive discussion on this subject.

\section{Preparational Material}\label{s2}

Two-fold preparation for validating Theorem \ref{t11} is presented through this intermediate section.

On the one hand, it is necessary to recall three fundamental properties on the variational $1<p<2$ capacity $\mathrm{pcap}(\bar{\Omega})$ and the logarithmic capacity (or conformal radius or transfinite diameter) $\mathrm{2cap}(\bar{\Omega})$ of a compact, convex, nonempty set $\bar{\Omega}\subset\mathbb R^2$ (cf. \cite{Lewis, CS, CC, Je96b, SZ}) determined by:
$$
\mathrm{pcap}(\bar{\Omega})= \lim_{|x|\to\infty}\begin{cases}
 2\pi{\Big(\frac{2-p}{p-1}\Big)^{p-1}}|x|^{2-p}\big(1-u_{\bar{\Omega}}(x)\big)^{p-1} \ \mathrm{as}\ \ p\in (1,2);\\
 \exp\big(\log|x|-u_{\bar{\Omega}}(x)\big)\ \ \mathrm{as}\ \ p=2,
\end{cases}
$$
where $d\mathcal{H}^2$ stands for the two-dimensional Hausdorff measure on $\mathbb R^2$ and $u_{\bar{\Omega}}$ is the solution of either ($\mathrm{eq}_{1<p<2}$) or ($\mathrm{eq}_{p=2}$).

Firstly, according to \cite[Lemma 2.16(a)]{CNSXYZ} for $p\in (1,2)$ and \cite[(6.3)]{Je96b} for $p=2$, we have: 
$$
(\star)\quad\mathrm{pcap}(\bar{\Omega})=
\begin{cases}
\int_{\partial\Omega}|\nabla u_{\bar{\Omega}}|^{p-1}\,d\mathcal{H}^1\ \ \mathrm{as}\ \ p\in (1,2);\\
\exp\left(2\pi\int_{\partial\Omega}\big(\log|\cdot|\big)|\nabla u_{\bar{\Omega}}(\cdot)|\,{d\mathcal{H}^1(\cdot)}\right)\ \ \mathrm{as}\ \ p=2.
\end{cases}
$$

Secondly, upon writing $A(\Omega)$ and $\mathrm{diam}(\Omega)$ for the $2$-dimensional Lebesgue measure of $\Omega$ and the diameter of $\Omega$, we have 
$$
\lim_{p\to 1}\mathrm{pcap}(\bar{\Omega})=\mathcal{H}^1(\partial\Omega)
$$ 
(cf. ($\star$) or \cite{LXZ, GH}) and the following isocapacitary/isodiametric inequalities (cf. \cite{X1, BPS, SZ} and their relevant references):
$$
 (\star\star)\quad \label{e4a}
  \left(\frac{A(\bar{\Omega})}{\pi}\right)^\frac12\le
  \begin{cases}
   \left(\frac{\mathrm{pcap}(\bar{\Omega})}{2\pi\big(\frac{p-1}{2-p}\big)^{1-p}}\right)^\frac1{2-p}\le 2^{-1}\mathrm{diam}(\bar{\Omega})\ \ \hbox{as}\ \ p\in (1,2);\\
   2^{-1}\mathrm{diam}(\bar{\Omega})\le 2\mathrm{pcap}(\bar{\Omega})\le\mathrm{diam}(\bar{\Omega})\ \ \hbox{as}\ \ p=2.
   \end{cases}
$$

Thirdly, if $h_{\bar{\Omega}}(x)=\sup_{y\in\bar{\Omega}}x\cdot y$ stands for the support function of $\bar{\Omega}$, then $(\star)$ can be formulated in the following way (cf. \cite[Theorem 1.1]{CNSXYZ} for $p\in (1,2)$ and \cite[Theorem 3.1]{X2} for $p=2$):
$$
(\star\star\star)\quad\int_{\partial\Omega}|\nabla u_{\bar{\Omega}}(x)|^px\cdot\mathsf{g}(x)\,d\mathcal{H}^1(x)=\begin{cases}
\Big(\frac{2-p}{p-1}\Big)\mathrm{pcap}(\bar{\Omega})\ \ \mathrm{as}\ \ p\in(1,2);\\
2\pi\ \ \mathrm{as}\ \ p=2.
\end{cases}
$$

On the other hand, three key lemmas and their arguments are needed.

\begin{lemma}\label{l21} Let $p\in (1,2]$ and $\Omega\subset\mathbb R^2$ be a bounded, convex, open set with non-empty interior. If $u_{\bar{\Omega}}$ is the $p$-equilibrium potential of $\Omega$ and there is an origin-centered open disk $\mathbb{D}(o,r)$ with radius $r>0$ such that $\Omega\subset \mathbb{D}(o,r)$, then there exists a constant $c>0$ depending only on $r$ such that $|\nabla u_{\bar{\Omega}}|\ge c$ almost everywhere on $\partial\Omega$ with respect to $d\mathcal{H}^1$.
\end{lemma}
\begin{proof} This follows directly from the case $n=2$ of both \cite[Lemma 2.18]{CNSXYZ} (for $p\in (1,2)$) and \cite[Theorem 3.2]{X2} (for $p=2$).
\end{proof}

\begin{lemma}\label{l22} For $p\in (1,2]$ and integer $m\ge 3$, a family $\{\zeta_j\}_{j=1}^m\subset\mathbb S^1$, and any point ${\mathsf{p}}\in\mathbb R^{m}$ with all nonnegative components ${\mathsf{p}}_1,...,{\mathsf{p}}_m$ let
$$
\begin{cases}
\Omega({\mathsf{p}})=\{x\in\mathbb R^2:\ x\cdot\zeta_j\le {\mathsf{p}}_j\ \forall\ j=1,...,m\};\\
{\mathsf M}=\{{\mathsf{p}}=({\mathsf{p}}_1,...,{\mathsf{p}}_m)\in\mathbb R^m:\ \mathrm{pcap}\big(\Omega({\mathsf{p}})\big)\ge 1\ \&\ {\mathsf{p}}_j\ge 0\ \forall\ j=1,...,m\}.
\end{cases}
$$
Given a sequence of $m$ positive numbers $\{c_j\}_{j=1}^m$, set
$
\Sigma({\mathsf{p}})=\sum_{j=1}^m c_j{\mathsf{p}}_j.
$
If $\{\zeta_j\}_{j=1}^m$ obeys the following three conditions:
\begin{itemize}
\item[{\rm(i)}] for any $\theta\in\mathbb S^{1}$ there is $j\in\{1,...,m\}$ such that $|\theta\cdot\zeta_j|>0$;

\item[{\rm(ii)}] $|\zeta_j+\zeta_k|>0\quad\forall\ j,k\in\{1,...,m\}$;

\item[{\rm(iii)}] $\sum_{j=1}^m c_j\zeta_j=0$.
\end{itemize}

\noindent Then there exists a point ${\mathsf{p}}^\ast\in\mathsf{M}$ such that:
\begin{itemize}

\item[{\rm(iv)}] $\inf_{{\mathsf{p}}\in\mathsf M}\Sigma({\mathsf{p}})=\Sigma({\mathsf{p}}^\ast)>0$;

\item[{\rm(v)}] $\Omega({\mathsf{p}}^\ast)$ is a polygon with $\{F_j\}_{j=1}^m$ and $\{\zeta_j\}_{j=1}^m$ as the only edges and outer unit normal vectors respectively;

\item[{\rm(vi)}] the $p$-equilibrium potential $u_{\Omega({\mathsf{p}}^\ast)}$ of ${\Omega({\mathsf{p}}^\ast)}$) obeys
$$
c_{1\le j\le m}=\tau_p^{-1}{\Sigma({\mathsf{p}}^\ast)}\int_{F_j}|\nabla u_{\Omega({\mathsf{p}}^\ast)}|^p\,d\mathcal{H}^1
$$
with
$$
\tau_p=\begin{cases} (2-p)(p-1)^{-1}\ \ \hbox{as}\ \ p\in (1,2);\\
2\pi\ \ \hbox{as}\ \ p=2.
\end{cases}
$$
\end{itemize}
\end{lemma}

\begin{proof} First of all, the argument for \cite[Theorem 5.4]{Je96a} is modified to reveal that $\Omega({\mathsf{p}})$ is closed and bounded thanks to (i) which derives
$$
|x|\le \sup_{j\in\{1,...,m\}, \theta\in\mathbb S^{1}}{{\mathsf{p}}_j}{|\theta\cdot\zeta_j|}^{-1}\quad{\forall}\quad x\in\Omega({\mathsf{p}}).
$$

Next, since $\{c_j\}_{j=1}^n$ is fixed and
$$
\begin{cases}
\Sigma({\mathsf{p}})\le\Big(\sum_{j=1}^m c_j^2\Big)^\frac12\Big(\sum_{j=1}^m{\mathsf{p}}_j^2\Big)^\frac12;\\
\inf_{{\mathsf{p}}\in\mathsf M}\Sigma({\mathsf{p}})<\infty,
\end{cases}
$$
each minimizing sequence for $\inf_{{\mathsf{p}}\in\mathsf M}\Sigma({\mathsf{p}})$ is bounded, and consequently, we can select a subsequence from the minimizing sequence that converges to ${\mathsf{p}}^\ast$. Now from the continuity of $\mathrm{pcap}(\cdot)$ under the Hausdorff distance $d_H(\cdot,\cdot)$ it follows that ${\mathsf{p}}^\ast\in\mathsf{M}$ is a minimizer. Of course,
$$
\begin{cases}
\hbox{pcap}\big(\Omega({\mathsf{p}}^\ast)\big)=1;\\
\inf_{{\mathsf{p}}\in\mathsf{M}}\Sigma({\mathsf{p}})=\Sigma({\mathsf{p}}^\ast)\quad\forall\quad p\in (1,2]. 
\end{cases}
$$

If $\Sigma({\mathsf{p}}^\ast)=0$, then ${\mathsf{p}}^\ast$ is the origin, and hence condition (i) implies that $\Omega({\mathsf{p}}^\ast)$ consists only of the origin, thereby yielding a contradiction $$1=\hbox{pcap}\big(\Omega({\mathsf{p}}^\ast)\big)=0.$$ 
So, (iv) holds.

Furthermore, if the interior $\big(\Omega({\mathsf{p}}^\ast)\big)^\circ$ of $\Omega({\mathsf{p}}^\ast)$ is empty, then (ii) can be used to deduce that $\Omega({\mathsf{p}}^\ast)$ is contained in a compact convex set $K$ with the Hausdorff dimension $\hbox{dim}_H(K)\le 1$. 

\begin{itemize}
	\item If $\hbox{dim}_H(K)=0$ then $\Omega({\mathsf{p}}^\ast)$ comprises one point and hence 
	$$
	0=\hbox{pcap}\big(\Omega({\mathsf{p}}^\ast)\big)=1,
	$$
	a contradiction.
	\item If $\hbox{dim}_H(K)=1$ then $\Omega({\mathsf{p}}^\ast)$ reduces to a segment and hence there exists $\zeta_j+\zeta_k=0$ for some $j,k\in\{1,...,m\}$ which is against the hypothesis (ii).
	
	\end{itemize}
Thus,
$\Omega({\mathsf{p}}^\ast)$ has a non-empty interior, and consequently (v) holds.

Finally, in order to check (vi), observe that ${\mathsf{p}}^\ast$ is not unique. Given $x_0\in\mathbb R^2$. If ${\mathsf{p}}\in\mathsf{M}$ then an application of (iii) implies that 
$${\mathsf{q}}=\big\{{\mathsf{p}}_j+x_0\cdot\zeta_j\big\}_{j=1}^m
$$ 
enjoys
$$
\begin{cases}
\Omega({\mathsf{q}})=x_0+\Omega({\mathsf{p}});\\ \Sigma({\mathsf{q}})=\Sigma({\mathsf{p}}).
\end{cases}
$$
Due to the fact that $\Omega({\mathsf{p}}^\ast)$ has non-empty interior, the origin may be translated to the interior of $\Omega({\mathsf{p}}^\ast)$ so that each component ${\mathsf{p}}_j^\ast$ is positive. Let $\mathsf{P}$ be the collection of those vectors ${\mathsf{p}}=(\mathsf{p}_1,...,\mathsf{p}_m)$ with
$$
\begin{cases}
{\mathsf{p}}_j\ge 0;\\
\Sigma\big(t{\mathsf{p}}+(1-t){\mathsf{p}}^\ast\big)=\Sigma({\mathsf{p}}^\ast)\ \ \forall\ \ t\in [0,1].
\end{cases}
$$
Then
$$
\begin{cases}
{\mathsf{p}}\in\mathsf{P};\\ t\Omega\big({\mathsf{p}}\big)+(1-t)\Omega({\mathsf{p}}^\ast)\subseteq\Omega\big(t{\mathsf{p}}+(1-t){\mathsf{p}}^\ast\big)\ \ \forall\ \ t\in [0,1],
\end{cases}
$$
plus \cite[Theorem 5.2]{CNSXYZ} (for $p\in (1,2)$) and \cite[(6.4)']{Je96b} or \cite[Theorem 4.4]{X2} (for $p=2$), ensures a constant $w_j>0$ such that
\begin{align*}
\sum_{j=1}^m({\mathsf{p}}_j-{\mathsf{p}}^\ast_j)w_j=\lim_{t\to 0}t^{-1}\Big({\hbox{pcap}\big(t\Omega({\mathsf{p}})+(1-t)\Omega({\mathsf{p}}^\ast)\big)-\hbox{pcap}\big(\Omega({\mathsf{p}}^\ast)\big)}\Big)\le 0.
\end{align*}
Whenever ${\mathsf{p}}$ is close to ${\mathsf{p}}^\ast=(\mathsf{p}_1^\ast,...,\mathsf{p}_m^\ast )$, the support function $h_{\Omega({\mathsf{p}})}$ of $\Omega({\mathsf{p}})$ enjoys
$$
h_{\Omega({\mathsf{p}})}(\zeta_j)={\mathsf{p}}_j\quad\forall\quad j\in \{1,...,m\}.
$$
Recall ${\mathsf{p}}^\ast_{j}>0$. So $$\sum_{j=1}^m({\mathsf{p}}_j-{\mathsf{p}}^\ast_j)w_j=0.$$ 
This last equation gives 
$$
w_j=\tau_p\big(\Sigma({\mathsf{p}}^\ast)\big)^{-1}c_j\quad\forall\quad j\in \{1,...,m\},
$$ 
thereby completing the proof.
\end{proof}

\begin{lemma}\label{l23} Let $p\in (1,2]$ and $\mu$ be a finite, nonnegative, Borel measure comprising a finite sum of point masses on $\mathbb S^{1}$ such that:
\begin{itemize}

\item[{\rm(i)}] $\mu$ is not supported on any equator of $\mathbb S^1$, i.e., $\inf_{\theta\in\mathbb S^1}\int_{\mathbb S^1}| \theta\cdot\xi|\,d\mu(\xi)>0$;

\item[{\rm(ii)}] $\mathrm{supp}(\mu)$ contains no any pair of antipodal points, i.e., if $\mu(\{\eta\})>0$ then $\mu(\{-\eta\})=0$;

\item[{\rm(iii)}] $\int_{\mathbb S^{1}}\theta\cdot\xi\,d\mu(\xi)=0\ \forall\
\theta\in\mathbb S^{1}$.
\end{itemize}
Then there exists a bounded, convex, nonempty, open polygon $O\subset\mathbb R^2$ such that
$d\mu_p(\bar{O},\cdot)=d\mu(\cdot)$.
\end{lemma}
\begin{proof} As in demonstrating \cite[Lemma 5.7]{Je96a}, we put
$$d\mu=\sum_{j=1}^m c_j\delta_{\zeta_j}$$
where $c_1,...,c_m>0$ are constants.
Note that conditions (i), (ii) and (iii) in Lemma \ref{l22} amount to (i), (ii) and (iii) in Lemma \ref{l23}, respectively. So, an application of Lemma \ref{l22} yields a bounded, convex, closed polygon ${P}$ containing the origin and a constant $c>0$ such that
$$\mathsf{g}_\ast(|\nabla u_{{P}}|^p\,d\mathcal{H}^1)=cd\mu.$$
Note that if $rP$ is the $r$-dilation of ${P}$ then
$$
\mathsf{g}_\ast(|\nabla u_{r{P}}|^p\,d\mathcal{H}^1)=r^{1-p}\mathsf{g}_\ast(|\nabla u_{{P}}|^p\,d\mathcal{H}^1).
$$
Thus, the desired result follows from choosing $r=c^\frac{1}{p-1}$ and $\bar{O}=rP$.
\end{proof}

\section{Proof of Theorem \ref{t11}}\label{s3}

(i) {\it Existence}. This comprises two parts.

{\it The if-part}. Suppose that $\mu$ has centroid at the origin and $\mathrm{supp}(\mu)$ does not comprise any pair of antipodal points. Of course, the first supposed condition is just
$$
\int_{\mathbb S^1}\theta\cdot\eta\,d\mu(\eta)=0\quad\forall\quad\theta\in\mathbb S^1.
$$
However, the second one implies that $\mu$ is not supported on any equator (the intersection of the unit circle $\mathbb S^1$ with any line through the origin) $\{\theta,-\theta\}$ of $\mathbb S^1$ where $\theta\in\mathbb S^1$ - otherwise $$\mathrm{supp}(\mu)=\{\theta_0,-\theta_0\}\ \ \mathrm{for\ some}\ \ \theta_0\in\mathbb S^1.
$$ 
Conversely, if $\mu$ is unsupported on any equator then $\mathrm{supp}(\mu)$ does not consist of any pair of antipodal points in $\mathbb S^1$ - otherwise there is $\theta_1\in\mathbb S^1$ such that $\mathrm{supp}(\mu)=\{\theta_1,-\theta_1\}$, i.e., $\mu$ is supported on an equator of $\mathbb S^1$. Consequently,
$$
0<\kappa\le\inf_{\theta\in\mathbb S^1}\int_{\mathbb S^{1}}|\theta\cdot\xi|\,d\mu(\xi).
$$

Using the above analysis, we may take a sequence $\{\mu_j\}_{j=1}^\infty$ of finite, nonnegative, Borel measures that are finite sums of point masses, not only converging to $\mu$ in the weak sense, but also satisfying (i)-(ii)-(iii) of Lemma \ref{l23}. According to Lemma \ref{l23}, for each $j$ there is a bounded, convex, closed set (polygon) $\bar{\Omega}_j\subset\mathbb R^2$ containing the origin such that the pull-back measure 
$$
d\mu_p(\bar{\Omega}_j,\cdot)=\mathsf{g}_\ast(|\nabla u_{\bar{\Omega}_j}|^p\,d\mathcal{H}^1)(\cdot)
$$ 
is equal to $d\mu_j(\cdot)$. On the one hand, by Lemma \ref{l21} and $(\star\star)$ there is a constant
${\kappa}_1>0$ independent of $j$ such that
\begin{equation*}\label{51}
{\kappa}_1\le\begin{cases}
   \Big(\big(\frac{p-1}{2-p}\big)^{p-1}\big({(2\pi)^{-1}\mathrm{pcap}(\bar{\Omega}_j)}\big)\Big)^\frac1{2-p}\ \ (\mathrm{for}\ 1<p<2)\\
   \mathrm{pcap}(\bar{\Omega}_j)\ (\mathrm{for}\ p=2)
   \end{cases}
\le\mathrm{diam}(\bar{\Omega}_j).
\end{equation*}
On the other hand, $\bar{\Omega}_j$ contains a segment $S_j$ such that its length is equal to $\mathrm{diam}(\bar{\Omega}_j)$. Due to the translation-invariance of $\mathrm{pcap}(\bar{\Omega}_j)$ it may be assumed that $S_j$ is the segment connecting $-2^{-1}\mathrm{diam}(\bar{\Omega}_j)\theta_j$ and $2^{-1}\mathrm{diam}(\bar{\Omega}_j)\theta_j$ where $\theta_1\in\mathbb S^1$. If $j$ is big enough, then
\begin{align*}
\int_{\mathbb S^1}h_{\bar{\Omega}_j}\,d\mu_j&\ge\int_{\mathbb S^1}h_{S_j}\,d\mu_j\\
&\ge 2^{-1}\mathrm{diam}(\bar{\Omega}_j)\int_{\mathbb S^1}|\theta_j\cdot\xi|\,d\mu_j(\xi)\\
&\ge 2^{-1}\mathrm{diam}(\bar{\Omega}_j)\kappa,
\end{align*}
and hence there is another constant $\kappa_2>0$ independent of $j$ such that 
$\kappa_2\ge\mathrm{diam}(\bar{\Omega}_j).
$
Hence, an application of the Blaschke selection principle (see e.g. \cite[Theorem 1.8.6]{Schn}) derives that $\{\bar{\Omega}_j\}_{j=1}^\infty$ has a subsequence, still denoted by $\{\bar{\Omega}_j\}_{j=1}^\infty$, which
converges to a bounded, compact, convex set $\bar{\Omega}_\infty\subset\mathbb R^2$ with $\mathrm{pcap}(\bar{\Omega}_\infty)>0$. In the sequel, we verify that the interior $(\bar{\Omega}_\infty)^\circ$ of $\bar{\Omega}_\infty$ is not empty. For this, assume $(\bar{\Omega}_\infty)^\circ=\emptyset$. Then the Hausdorff dimension $\hbox{dim}_H(\bar{\Omega}_\infty)$ of $\bar{\Omega}_\infty$ is strictly less than $2$. If $\hbox{dim}_H(\bar{\Omega}_\infty)=0$ then the convexity of $\bar{\Omega}_\infty$ ensures that $\bar{\Omega}_\infty$ is a single point and hence $\hbox{pcap}(\bar{\Omega}_\infty)=0$, contradicting $\hbox{pcap}(\bar{\Omega}_\infty)>0$. This illustrates $\hbox{dim}_H(\bar{\Omega}_\infty)=1$. Consequently, there exists a constant $\kappa_3>0$ and a point $\xi\in\mathbb S^{1}$ such that the pull-back measure $\mathsf{g}_\ast(d\mathcal{H}^1|_{\partial\Omega_\infty})$ of $\mathcal{H}^1|_{\partial\Omega_\infty}$ to $\mathbb S^1$ via the Gauss map $\mathsf{g}$ is equal to $\kappa_3(\delta_\xi+\delta_{-\xi})$.
Upon using Lemma \ref{l21} we obtain a positive constant $\kappa_4$ (independent of $j$ but dependent of $p$ and the radius of an appropriate $o$-centered ball containing all $\bar{\Omega}_j$) such that
$|\nabla u_{\bar{\Omega}_j}|^p\ge\kappa_4$ holds almost everywhere on $\partial{\Omega}_j$. Suppose that $f\in C(\mathbb S^{1})$ (the class of all continuous functions on $\mathbb S^1$) is positive and its support is contained in a small neighbourhood $N(\xi)\subset\mathbb S^{1}$ of $\xi\in\mathbb S^1$ only. Now, we use Fatou's lemma to derive
\begin{align*}
\int_{N(\xi)}f\,d\mu&=\liminf_{j\to\infty}\int_{\mathbb S^{1}}f\,d\mu_j\\
&\ge\kappa_4\liminf_{j\to\infty}\int_{\mathbb S^{1}}f\,\mathsf{g}_\ast\Big(d\mathcal{H}^{1}|_{\partial \bar{\Omega}_j}\Big)\\
&\ge\kappa_4\int_{N(\xi)}\liminf_{j\to\infty}\mathsf{g}_\ast\Big(d\mathcal{H}^{1}|_{\partial \bar{\Omega}_j}\Big)\\
&=\kappa_4\int_{N(\xi)}f\,\mathsf{g}_\ast\Big(d\mathcal{H}^{1}|_{\partial\Omega_\infty}\Big)\\
&=\kappa_4 f(\xi).
\end{align*}
Thus, Radon-Nikodym's differentiation of $\mu$ with respect to the Dirac measure concentrated at $\xi$ (cf. \cite[page 42, Theorem 3]{EvaG}) implies that $\mu$ must have a positive mass at $\xi$, and similarly, $\mu(\{-\xi\})>0$. Thus $$\mathrm{supp}(\mu)\supset\{\xi,-\xi\}.$$
Meanwhile, if $$\xi_0\in\mathbb S^1\setminus\{\xi,-\xi\},$$ 
then an application of the fact that the polygon $\bar{\Omega}_j$ (whose Gauss map is denoted by $\mathsf{g}_j:\partial{\Omega}_j\to\mathbb S^1$) approaches $\Omega$ (which has only two outer unit normal vectors $\pm\xi$) ensures that $\xi_0$ is not in the set of all outer unit normal vectors of $\bar{\Omega}_j$, thereby yielding 
$$
\mathcal{H}^1(\mathsf{g}_j^{-1}\big(\{\xi_0\})\big)=0\ \ \mbox{as}\ \  j>N
$$ 
for a sufficiently large $N$. According to \cite[Theorems 1\&3]{LN}, there is $q>p$ such that $|\nabla u_{\bar{\Omega}_j}|^q$ is integrable on $\mathsf{g}_j^{-1}(\{\xi_0\})$ with respect to $d\mathcal{H}^1|_{\partial{\Omega}_j}$. This existence, the H\"older inequality, the weak convergence of $\mu_j$, and Fatou's lemma, imply
\begin{align*}
0&\le\mu\big(\{\xi_0\}\big)\\
&\le\liminf_{j\to\infty}\mu_j\big(\{\xi_0\}\big)\\
&=\liminf_{j\to\infty}\int_{\mathsf{g}_j^{-1}(\{\xi_0\})}|\nabla u_{\bar{\Omega}_j}|^p\,d\mathcal{H}^1|_{\partial{\Omega}_j}\\
&\le\liminf_{j\to\infty}\left(\int_{\mathsf{g}_j^{-1}(\{\xi_0\})}|\nabla u_{\bar{\Omega}_j}|^q\,d\mathcal{H}^1|_{\partial{\Omega}_j}\right)^\frac{p}q\left(\mathcal{H}^1\big(\mathsf{g}_j^{-1}(\{\xi_0\})\big)\right)^{1-\frac{p}{q}}\\
&=0.
\end{align*}
Consequently, $\mu(\{\xi_0\})=0$. So, $$\mathrm{supp}(\mu)=\{\xi,-\xi\},$$ 
which contradicts the second supposed condition. In other words, $(\bar{\Omega}_\infty)^\circ\not=\emptyset$. This, along with 
$$
d\mu_p(\bar{\Omega}_j,\cdot)=d\mu_j(\cdot)
$$
and the weak convergence of $\mu_j\to\mu$, derives 
$$
d\mu_p(\bar{\Omega}_\infty,\cdot)=d\mu(\cdot),
$$
as desired.

{\it The only-if part}. Suppose that $d\mu_p(\bar{\Omega},\cdot)=d\mu(\cdot)$ holds for a bounded, convex, nonempty, open set $\Omega\subset\mathbb R^2$. Note first that $\hbox{pcap}(\cdot)$ is translation invariant. So
$$
\hbox{pcap}(\bar{\Omega}+\{x_0\})=\hbox{pcap}(\bar{\Omega})\quad\forall\quad x_0\in\mathbb R^2.
$$
However, the translation $\bar{\Omega}\mapsto \bar{\Omega}+\{x_0\}$ changes $x\cdot\mathsf{g}$ to
$x\cdot{\mathsf g}+x_0\cdot x.$ Thus, an application of $(\star\star\star)$ yields
\begin{align*}
&\int_{\partial(\bar{\Omega}+\{x_0\})}\big(x\cdot\mathsf{g}(x)\big)|\nabla u_{\bar{\Omega}+\{x_0\}}(x)|^p\,d\mathcal{H}^{1}(x)\\
&\quad\quad=\int_{\partial\Omega}\big(x_0\cdot\mathsf{g}(x)\big)|\nabla u_{\bar{\Omega}}(x)|^p\,d\mathcal{H}^{1}(x)+
\int_{\partial\Omega}\big(x\cdot\mathsf{g}(x)\big)|\nabla u_{\bar{\Omega}}(x)|^p\,d\mathcal{H}^{1}(x).
\end{align*}
Consequently,
$$
\int_{\partial \Omega}(x_0\cdot\mathsf{g}(x))|\nabla u_{\bar{\Omega}}(x)|^p\,d\mathcal{H}^{1}(x)=0.
$$
This in turn implies the following linear constraint on $\mu$:
$$
\int_{\mathbb S^{1}}\theta\cdot\eta\,d\mu(\theta)=\int_{\mathbb S^{1}}\theta\cdot\eta\,d\mu_p(\bar{\Omega},\theta)=0\ \ \forall\ \ \eta\in\mathbb S^{1}.
$$

Next, let us validate that $\mathrm{supp}(\mu)$ does not comprise any pair of antipodal points. If this is not true, then there is $\theta_0\in\mathbb S^1$ such that $$\mathrm{supp}(\mu)=\{\theta_0,-\theta_0\}.$$ 
However, the following considerations (partially motivated by \cite[Lemma 4.1]{BT} handling the necessary part of a planar $L_p$-Minkowski problem from \cite{Lut}) will show that this last identification cannot be valid.

{\it Case $o\in\Omega$}. This, together with Lemma \ref{l21}, ensures 
$$
\{\theta_0,-\theta_0\}=\mathrm{supp}\big(\mu_p(\bar{\Omega},\cdot)\big)=\mathrm{supp}\big(\mathsf{g}_\ast(d\mathcal{H}^1|_{\partial\Omega})\big).
$$
However, $\bar{\Omega}$ is not degenerate, so $\mathrm{supp}\big(\mathsf{g}_\ast(d\mathcal{H}^1|_{\partial\Omega}))$ cannot be $\{\theta_0,-\theta_0\}$ - a contradiction occurs.

{\it Case $o\in\partial\Omega$}. Denote by $\Lambda$ the exterior normal cone at $o$ such that
$$
\Lambda\cap\mathbb S^1=\big\{\eta\in\mathbb S^1: \ h_\Omega(\eta)=0\big\}.
$$  
Since 
$\mathrm{supp}(\mu)$ coincides with $\mathrm{supp}(\mu_p(\bar{\Omega},\cdot))$,
it follows that $h_\Omega(\theta_0)$ and $h_\Omega(\theta_0)$ are positive. This in turn implies that $\pm\theta_0$ are not in $\Lambda$. Without loss of generality we may assume that $\Lambda\cap\mathbb S^1$ is a subset of the following semi-circle
$$
\mathbb T(-\theta_0,o)=\big\{\zeta\in\mathbb S^1:\ \zeta\cdot\theta_0<0\big\}.
$$
Accordingly, if 
$$
\eta\in\mathbb T(\theta_0,o)=\big\{\zeta\in\mathbb S^1:\ \zeta\cdot\theta_0>0\big\},
$$
then $h_{\bar{\Omega}}(\eta)>0$. Also because of 
$$
\mathsf{g}_\ast\Big(d\mathcal{H}^1|_{\partial\Omega}\Big)\big(\mathbb T(\theta_0,o)\big)>0
$$
and Lemma \ref{l21} (with a positive constant $c$ depending only on $p$ and $r$ - the radius of a suitable ball $\mathbb{D}(o,r)\supset\Omega$), we utilize $$\mathrm{supp}(\mu)=\{\theta_0,-\theta_0\}$$ to obtain the following contradictory computation:
\begin{align*}
0&=\mu\big(\mathbb T(\theta_0,o)\big)\\
&=\mu_p\big(\bar{\Omega},\mathbb T(\theta_0,o)\big)\\
&=\int_{\mathsf{g}^{-1}\big(\mathbb T(\theta_0,o)\big)}|\nabla u_{\bar{\Omega}}|^p\,d\mathcal{H}^1\\
&\ge c^p\mathcal{H}^1\big(\mathsf{g}^{-1}\big(\mathbb T(\theta_0,o)\big)\\
&>0.
\end{align*}

(ii) {\it Uniqueness}. Suppose that ${\Omega}_0, {\Omega}_1$ are two solutions of the equation $d\mu_p(\bar{\Omega},\cdot)=d\mu(\cdot)$. Then
$$
g_\ast(|\nabla u_{\bar{\Omega}_0}|^p\,d\mathcal{H}^1)=g_\ast(|\nabla u_{\bar{\Omega}_1}|^p\,d\mathcal{H}^1).
$$
To reach the conclusion that ${\Omega}_0$ and ${\Omega}_1$ are the same up to a translate, we define 
$$
[0,1]\ni t\mapsto f_p(t)=\begin{cases}\Big(\hbox{pcap}\big((1-t)\bar{\Omega}_0+t\bar{\Omega}_1\big)\Big)^\frac{1}{2-p}\ \ \mathrm{as}\ \ p\in (1,2);\\
\hbox{pcap}\big((1-t)\bar{\Omega}_0+t\bar{\Omega}_1\big)\ \ \mathrm{as}\ \ p=2,
\end{cases}
$$ 
and handle the following two cases.

{\it Case $p\in (1,2)$}. In a manner (cf. \cite{CJeL}) slightly different from proving \cite[Theorem 1.2]{CNSXYZ} (under $n=2>p>1$), we use the chain rule, \cite[Theorem 1.1]{CNSXYZ} (under $n=2$) and ($\star\star\star$) to get
\begin{align*}
f'_p(0)&=\frac{\big(f_p(0)\big)^{p-1}}{(\frac{2-p}{p-1})}
\int_{\partial \bar{\Omega}_0}\big(h_{\bar{\Omega}_1}(\mathsf{g})-h_{\bar{\Omega}_0}(\mathsf{g})\big)|\nabla u_{\bar{\Omega}_0}|^p\,d\mathcal{H}^1\\
&=\frac{\big(f_p(0)\big)^{p-1}}{(\frac{2-p}{p-1})}
\left(\int_{\partial \bar{\Omega}_0}h_{\bar{\Omega}_1}(\mathsf{g})|\nabla u_{\bar{\Omega}_0}|^p\,d\mathcal{H}^1-\int_{\partial{\Omega}_0}h_{\bar{\Omega}_0}(\mathsf{g})|\nabla u_{\bar{\Omega}_0}|^p\,d\mathcal{H}^1\right)\\
&=\frac{\big(f_p(0)\big)^{p-1}}{(\frac{2-p}{p-1})}
\left(\int_{\mathbb S^1}h_{\bar{\Omega}_1}\mathsf{g}_\ast\big(|\nabla u_{\bar{\Omega}_0}|^p\,d\mathcal{H}^1\big)-\int_{\mathbb S^1}h_{\bar{\Omega}_0}\mathsf{g}_\ast\big(|\nabla u_{\bar{\Omega}_0}|^p\,d\mathcal{H}^1\big)\right)\\
&=\frac{\big(f_p(0)\big)^{p-1}}{(\frac{2-p}{p-1})}
\left(\int_{\mathbb S^1}h_{\bar{\Omega}_1}\mathsf{g}_\ast\big(|\nabla u_{\bar{\Omega}_1}|^p\,d\mathcal{H}^1\big)-\int_{\mathbb S^1}h_{\bar{\Omega}_0}\mathsf{g}_\ast\big(|\nabla u_{\bar{\Omega}_0}|^p\,d\mathcal{H}^1\big)\right)\\
&=\big(f_p(0)\big)^{p-1}\Big(\big(f_p(1)\big)^{2-p}-\big(f_p(0)\big)^{2-p}\Big).
\end{align*}
According to \cite[Theorem 1]{CS}, $f_p$ is concave, and so
$$
f_p(1)-f_p(0)\le f'_p(0)=\big(f_p(0)\big)^{p-1}\Big(\big(f_p(1)\big)^{2-p}-\big(f_p(0)\big)^{2-p}\Big).
$$
This, along with exchanging $\bar{\Omega}_0$ and $\bar{\Omega}_1$, implies
$$
\mathrm{pcap}(\bar{\Omega}_1)=f_p(1)\le f_p(0)=\mathrm{pcap}(\bar{\Omega}_0)\le f_p(1)=\mathrm{pcap}(\bar{\Omega}_1),
$$
thereby producing $f_p'(0)=0$ and $f_p$ being a constant thanks to the concavity of $f_p$. Since $\bar{\Omega}_0$ and $\bar{\Omega}_1$ have the same $p$-capacity, an application of the equality in \cite[Theorem 1]{CS} yields that ${\Omega}_0$ is a translate of ${\Omega}_1$.

{\it Case $p=2$}. Referring to the argument for \cite[Theorem 5.1]{X2} under $n=2$, we employ \cite[Theorems 4.4 \& 3.1]{X2} to deduce
\begin{align*}
f_2'(0)&=(2\pi)^{-1}f_2(0)\int_{\partial \bar{\Omega}_0}\big(h_{\bar{\Omega}_1}(\mathsf{g})-h_{\bar{\Omega}_0}(\mathsf{g})\big)|\nabla u_{\bar{\Omega}_0}|^2\,d\mathcal{H}^1\\
&=(2\pi)^{-1}f_2(0)\left(\int_{\partial \bar{\Omega}_0}h_{\bar{\Omega}_1}(\mathsf{g})|\nabla u_{\bar{\Omega}_0}|^2\,d\mathcal{H}^1-2\pi\right)\\
&={(2\pi)}^{-1}f_2(0)\left(\int_{\mathbb S^{1}}h_{\bar{\Omega}_1}\,\mathsf{g}_\ast(|\nabla u_{\bar{\Omega}_0}|^2\,d\mathcal{H}^1)-2\pi\right)\\
&={(2\pi)}^{-1}f_2(0)\left(\int_{\mathbb S^{1}}h_{\bar{\Omega}_1}\,\mathsf{g}_\ast(|\nabla u_{\bar{\Omega}_1}|^2\,d\mathcal{H}^1)-2\pi\right)\\
&={(2\pi)}^{-1}f_2(0)(2\pi-2\pi)\\
&=0.
\end{align*}
Note that $t\mapsto f_2(t)$ is concave on $[0,1]$ (cf. \cite{Bor, CC}). So $f_2$ is a constant function on $[0,1]$, in particular, we have
\begin{equation*}\label{ef}
\hbox{2cap}(\bar{\Omega}_1)=f_2(1)=f_2(t)=f_2(0)=\mathrm{2cap}(\bar{\Omega}_0).
\end{equation*}
 As a consequence, the equation
 $$
 f_2(t)=f_2(0)\ \ \forall\ \ t\in [0,1]
 $$ 
 and \cite[Theorem 3.1]{CC} ensure that ${\Omega}_0$ and ${\Omega}_1$ are the same up to translation and dilation. But, 
 $$\mathrm{2cap}(\bar{\Omega}_0)=\mathrm{2cap}(\bar{\Omega}_1)
 $$ forces that ${\Omega}_1$ is only a translate of ${\Omega}_0$.

(iii) {\it Regularity}. \cite[Theorem 1.4]{CNSXYZ} covers the case $1<p<2=n$. The argument for \cite[Theorem 1.4]{CNSXYZ} or for the regularity part of \cite[Theorem 0.7]{Je96a} (cf. \cite[Theorem 7.1]{Je96a} and \cite{GH}) under $n=2$ can be modified to verify the case $p=2$. For reader's convenience, an outline of this verification under $p\in (1,2]$ is presented below.

Firstly, we observe that Lemmas 7.2-7.3-7.4 in \cite{CNSXYZ} are still valid for the $(1,2]\ni p$-equilibrium potential $u_{\bar{\Omega}}$.

Secondly, \cite[Lemma 6.16]{Je96a} can be used to produce two constants $c>0$ and $\epsilon\in (0,1)$ (depending on the Lipschitz constant of $\Omega$) such that (cf. \cite[Lemma 7.5]{CNSXYZ} for $p\in (1,2)$ and \cite[Theorem 6.5]{Je96a} for $p=2$)
\begin{equation*}\label{eM1}
\int_{H\cap\partial \Omega}\big(\delta(\cdot,H\cap\partial \Omega)\big)^{1-\epsilon}|\nabla u_{\bar{\Omega}}(\cdot)|^p\,d\mathcal{H}^1(\cdot)\le c\mathcal{H}^1(H\cap\partial\Omega)\inf_{H\cap\partial \Omega}|\nabla u_{\bar{\Omega}}|^p
\end{equation*}
holds for any half-plane 
$H\subset\mathbb R^2$ with
$H\cap{\mathbb D}(o,r_{int})=\emptyset,
$ 
where $r_{int}$ is the inner radius of $\Omega$, and $\delta(x,H\cap\partial\Omega)$ is a normalized distance from $x$ to $H\cap\partial\Omega$.

Thirdly, from \cite[Lemma 7.7]{CNSXYZ} it follows that if
$$d\mu_p(\bar{\Omega},\cdot)=\psi(\cdot)\, d\ell(\cdot)$$ is valid for some integrable function $\psi$ being greater than a positive constant $c$ on $\mathbb S^{1}$, and if $\phi$ stands for the convex and Lipschitz function defined on a bounded open interval $O\subset\mathbb R^{1}$ whose graph 
$$
G=\{(s,\phi(s)):\ s\in O\}
$$ 
is a portion of the convex curve $\partial\Omega$, then $\phi$ enjoys the following $(1,2]\ni p$-Monge-Amp\'ere equation in Alexandrov's sense (cf. \cite[p.6]{Gut}):
\begin{align*}
\phi''(s)&=\mathrm{det}\big(\nabla^2\phi(s)\big)\\
&={\big(1+|\nabla\phi(s)|^2\big)^\frac{3}{2}\big|\big(\nabla u_{\bar{\Omega}}\big)\big(s,\phi(s)\big)\big|^p}\big(\psi(\xi)\big)^{-1}\\
&={\Big(1+\big(\phi'(s)\big)^2\Big)^\frac{3}{2}\big|\big(\nabla u_{\bar{\Omega}}\big)\big(s,\phi(s)\big)\big|^p}\big(\psi(\xi)\big)^{-1}\\
&\equiv\Phi_p(\bar\Omega,s),
\end{align*}
where 
\begin{itemize}
\item[$\circ$]
$$
\frac{d}{ds}u_{\bar{\Omega}}\big(s,\phi(s)\big)=\big(1,\phi'(s)\big)\cdot\big(\nabla u_{\bar{\Omega}}\big)\big(s,\phi(s)\big)
$$
is utilized to explain the action of $\nabla u_{\bar{\Omega}}$ at $\big(s,\phi(s)\big)\in G$;
\item[$\circ$]
$$
s\mapsto\phi''(s)\Big(1+\big(\phi'(s)\big)^2\Big)^{-\frac{3}{2}}\big|\big(\nabla u_{\bar{\Omega}}\big)\big(s,\phi(s)\big)\big|^{-p}
$$
is regarded as the $p$-equilibrium-potential-curvature on $G\subset\partial\Omega$;
\item[$\circ$]
$$
\xi={\big(\phi'(s),-1\big)}\Big({1+\big(\phi'(s)\big)^2}\Big)^{-\frac12}
$$
is written for the outer unit normal vector at $\big(s,\phi(s)\big)\in G$.
\end{itemize}

Fourthly, an application of the secondly-part and the thirdly-part above and \cite[Theorem 7.1]{Je96a} derives that if $\psi$ is bounded above and below by two positive constants then Caffarelli's methodology developed in \cite{Caf3} can be adapted to establish that $\partial\Omega$ is of $C^{1,\epsilon}$ for the above-found $\epsilon\in (0,1)$. Now, for $\alpha\in (0,1)$ let the positive function $\psi$ in $$d\mu_p(\bar{\Omega},\cdot)=\psi(\cdot) d\ell(\cdot)$$ belong to $C^{0,\alpha}(\mathbb S^{1})$. Since $\partial\Omega$ is of $C^{1,\epsilon}$, a barrier argument, plus \cite{Lie}, yields that $|\nabla u_{\bar{\Omega}}|$ is not only bounded above and below by two positive constants (and so is $\phi''$ on $O$), but also $|\nabla u_{\bar{\Omega}}|$ is of $C^{0,\epsilon}$ up to $\partial\Omega$. From the thirdly-part above it follows that 
$\Phi_p(\bar{\Omega},\cdot)$ is of $C^{0,\epsilon_1}$ for some $\epsilon_1\in (0,1)$. This, along with  $\phi''(\cdot)=\Phi_p(\bar\Omega,\cdot)$, gives that $\phi$ is of $C^{2,\epsilon_1}$. As a consequence, we see that $|\nabla u_{\bar{\Omega}}|$ is of $C^{1,\epsilon_2}$ up to $\partial\Omega$ for some $\epsilon_2\in (0,1)$, and thereby finding that 
$\Phi_p(\bar\Omega,\cdot)$ is of $C^{0,\alpha}$. Accordingly, $\partial\Omega$ being of $C^{2,\alpha}$ follows from Caffarelli's three papers \cite{Caf0,Caf1,Caf2}. Continuing this initial precess, we can reach the desired higher order regularity.

\medskip
\noindent {\it Acknowledgement}. The author is grateful to Han Hong and Ning Zhang for several discussions on the only-if-part of Theorem \ref{t11}(i).

\end{document}